\documentclass[12pt]{amsart}

\usepackage{amssymb,amsthm}
\usepackage{amsmath}

\newtheorem{thm}{Theorem}

\newtheorem{theorem}{Theorem}[section]
\newtheorem{corollary}[theorem]{Corollary}

\newtheorem{lemma}[theorem]{Lemma}

\def\irr#1{{\rm  Irr}(#1)}

\begin{document}

\title{$p$-group Camina pairs}

\author{Mark L. Lewis}

\address{Department of Mathematical Sciences, Kent State University, Kent, OH 44242}

\email{lewis@math.kent.edu}


\begin{abstract}
Let $(G,Z(G))$ be a Camina pair.  We prove that $G$ must be a $p$-group for some prime $p$.  We also prove that $|Z(G)| < |G:Z(G)|^{3/4}$.  Also, we discuss how one might build examples with $|Z(G)| > |G:Z(G)|^{1/2}$, although we are not able to prove the existence of such examples.
\end{abstract}

\maketitle

\section{Introduction}

Throughout this note $G$ is a finite group.  A pair $(G,N)$ is a {\it Camina pair} if $1 < N < G$ is a normal subgroup of $G$ and for every element $g \in G \setminus N$, the element $g$ is conjugate to all of $gN$.  These have been studied in a number of places (\cite{Camina}, \cite{ChMc}, \cite{ChMaSc}, \cite{coprime}, and \cite{sylow}).  There are a number of equivalent conditions for Camina pairs.  An equivalent condition that we use is: $(G,N)$ is a Camina pair if and only if for every element $g \in G \setminus N$ and for every element $n \in N$, there exists an element $y \in G$ so that $[y,g] = n$.  Another condition we refer to is: $(G,N)$ is a Camina pair if and only if every character in $\irr {G \mid N}$ vanishes on $G \setminus N$.  A third condition we need is: $(G,N)$ is a Camina pair if and only if $|C_G (g)| = |C_{G/N} (gN)|$ for all $g \in G \setminus N$.

It follows easily from the definition that if $(G,N)$ is a Camina pair, then $Z(G) \le N \le G'$.  The two extreme cases are when $N = G'$ and when $N = Z (G)$.  In the case $N
= G'$, we say that the group $G$ is a Camina group, and this case has been studied in a number of papers (\cite{DaSc}, \cite{MacD1}, and \cite{more}).  In this paper, we consider the other extreme case where $N = Z (G)$.

Under the assumption that $(G,Z(G))$ is a Camina pair, we will give a very short proof that $G$ must be a $p$-group for some prime $p$.  Our main goal in this paper is to bound $|Z(G)|$ in terms of $|G:Z(G)|$ and associated quantities.  It is quite easy to show that $|Z (G)| < |G: Z(G)|$ (see Theorem \ref{exponent quotient}).  We will show that better bounds exist.  The first bound in terms of the index of the derived subgroup is not much more difficult to prove.

\begin{thm} \label{first}
Let $(G,Z(G))$ be a Camina pair.  Then $|Z(G)| \le |G:G'|$.
\end{thm}

When $(G,Z(G))$ is a Camina pair and $Z (G) < G'$, then we also obtain a bound for $|Z (G)|$ in terms of $|G':Z(G)|$.  In particular, we have the following result.

\begin{thm} \label{second}
Let $(G,Z(G))$ be a Camina pair with $Z(G) < G'$.  Then $|Z (G)| < |G':Z(G)|^3$.
\end{thm}

With this theorem in hand, we are able to prove the most general result in this note.

\begin{thm} \label{third}
Let $(G,Z(G))$ be a Camina pair.  Then $|Z (G)| < |G:Z(G)|^{3/4}$.
\end{thm}

When $|G:Z(G)|$ is small, we can obtain a stronger bound.  This is the content of the next theorem.

\begin{thm} \label{fourth}
Let $(G,Z(G))$ be a Camina pair with $Z (G) < G'$.  Then either $|Z (G)| \le |G:Z(G)|^{1/2}$ or $|Z (G)|p^4 \le |G:Z(G)|$.  In particular, if $|G:Z(G)| \le p^8$, then $|Z (G)| \le |G:Z (G)|^{1/2}$.
\end{thm}

Also, if $G/ Z(G)$ does not have exponent $p$, we can also obtain a stronger result.

\begin{thm} \label{fifth}
Let $(G,Z(G))$ be a Camina pair.  If $G$ is a $p$-group and the exponent of $G/Z (G)$ is not $p$, then $|Z (G)| < |G:Z(G)|^{1/2}$.
\end{thm}

When we originally started working on this problem, we conjectured that $|Z (G)| \le |G:Z (G)|^{1/2}$ whenever $(G,Z(G))$ is a Camina pair.  This conjecture was based on the fact that this is known to be true if $G$ is a Camina group of nilpotence class 2.  (See Theorem 3.2 of \cite{MacD1}.)  At this time, we do not have any counter-examples to this conjecture.  However, we will sketch out some properties of possible groups that would violate the conjecture, and while we cannot prove that such groups exist, we do have seem reason to believe that they do.

\section{Basics}

We first prove that if $(G,Z(G))$ is a Camina pair, then $G$ is a $p$-group for some prime $p$.  The key to proving this result is a fact proved by Camina.

\begin{lemma}
Let $(G,Z(G))$ be a Camina pair.  Then $G$ is a $p$-group for some prime $p$.
\end{lemma}

\begin{proof}
In Theorem 2 of \cite{Camina}, Camina proved that if $(G,N)$ is a Camina pair, then $G$ is a Frobenius group with Frobenius kernel $N$ or either $G/N$ or $N$ is a $p$-group for some prime $p$.  Since $N = Z(G)$, we see that $G$ is not a Frobenius group.  So either $G/Z(G)$ or $Z(G)$ is a $p$-group for some prime $p$.  Given $a \in G \setminus Z(G)$ and $z \in Z(G)$, we know that $a$ and $az$ are conjugate, so they have the same order.  Since $a$ and $z$ commute, this implies that $o(z)$ divides $o(a)$.  If $G$ is not a $p$-group, then one can choose $a$ and $z$ to have coprime orders, a contradiction.
\end{proof}

When $(G,N)$ is a Camina pair and $G$ is a $p$-group for some prime $p$, there is a connection between $N$ and the lower and upper central series for $G$.  We denote the lower central series by $G_1 = G$, $G_2 = [G,G] = G'$, and $G_i = [G_{i-1},G]$ for all $i \ge 2$, and we denote the upper central series by $Z_1 = Z(G)$ and $Z_i/Z_{i-1} = Z(G/Z_{i-1})$ for $i \ge 2$.  The next two lemmas were proved by Macdonald in \cite{MacD1}.  This first lemma (Theorem 2.2 of \cite{MacD1}) shows that if $(G,Z(G))$ is a Camina pair, then each factor in the upper central series has exponent $p$.

\begin{lemma} \cite{MacD1} \label{exponent center}
Let $(G,Z(G))$ be a Camina pair, and let $p$ be the prime so that $G$ is a $p$-group.  Then $Z_i/Z_{i-1}$ has exponent $p$ for $1 \le i \le c$ where $c$ is the nilpotence class of $G$.  In particular, $Z(G)$ is an elementary abelian $p$-group.
\end{lemma}

This second lemma (Lemma 2.1 of \cite{MacD1}) was proved by Macdonald in the more general setting where $(G,N)$ is a Camina pair and $G$ is a $p$-group.  He actually proved that $N$ is a term in both the upper and lower central series.  When $N = Z(G)$, this reduces to the following.

\begin{lemma} \cite{MacD1}
Let $(G,Z(G))$ be a Camina pair.  If $G$ has nilpotence class $c$, then $Z(G) = G_c$.
\end{lemma}

Suppose $N$ is a normal subgroup of $G$.  The character $\theta \in \irr N$ is said to be {\it fully ramified} with respect to $G/N$ if $\theta$ is $G$-invariant and $\theta^G$ has a unique irreducible constituent.  Problem 6.3 of \cite{text} shows that this there are two conditions equivalent to this conditions.  We make use of both of these equivalent conditions in the next lemma which is well-known.

\begin{lemma}
Let $(G,Z(G))$ be a Camina pair.  Every character in $\irr {G \mid Z(G)}$ is fully ramified with respect to $G/Z(G)$.  In particular, $|G:Z(G)|$ is a square.
\end{lemma}

\begin{proof}
Since $(G,Z(G))$ is a Camina pair, we know that every character in $\irr {G \mid Z(G)}$ vanishes on $G \setminus Z(G)$.  In Problem 6.3 of \cite{text}, it is shown that vanishing on $G \setminus N$ and being homogeneous upon restriction to $N$ implies that an irreducible character is fully ramified with respect to $G/N$.  Also, it is shown in that problem that if $\chi$ is fully ramified, then $|G: Z(G)| = \chi (1)^2$.
\end{proof}

\section{Exponent bigger than $p$}

Recall that if $(G,G')$ is a Camina pair, then $G$ is a Camina group.  When $G$ is a Camina $p$-group of nilpotence class 2, then the following result was proved by Macdonald in Theorem 3.2 of \cite{MacD1}.

\begin{lemma} \label{Camina gp}
Let $G$ be a Camina group of nilpotence class 2.  Then $|Z(G)|^2 \le |G:Z (G)|$.
\end{lemma}

Recall that if $(G,Z(G))$ is a Camina pair, then $Z(G) \le G'$. Thus, we may assume that $Z(G) < G'$.

Let $g \in G \setminus Z(G)$.  We set $D_G (g) = \{ x \in G \mid [g,x] \in Z(G) \}$.  It is not very hard to see that $C_G (g) \le D_G (g)$ and $D_G (g)/C_G (g) = C_{G/Z(G)} (gZ(G))$.  It follows that $D_G (g)$ is a subgroup of $G$.  When the group $G$ is clear, we write $C (g)$ for $C_G (g)$ and $D (g)$ for $D_G (g)$.

\begin{lemma}\label{cents}
Let $(G,Z(G))$ be a Camina pair where $G$ is a $p$-group.  If $g \in G \setminus Z(G)$, then $D (g)/C (g) \cong Z(G)$.  In particular, $D (g)' \le C (g)$.
\end{lemma}

\begin{proof}
Consider the map $D (g) \rightarrow Z(G)$ by $d \mapsto [d,g]$.  By the definition of $D (g)$, we know that this map is well-defined.  Since $Z (G)$ is central, this map is a homomorphism.  To see that this map is onto, suppose $z \in Z (G)$.  By the Camina condition, there exists $y \in G$ so that $[y,g] = z$.  But now, $[y,g] \in Z(G)$ implies that $y \in D (g)$.  Finally, observe that $C (g)$ is the kernel of this homomorphism.  We now apply the first isomorphism theorem to see that $D (g)/C (g) \cong Z(G)$.
\end{proof}

We now obtain a bound of $|Z (G)|$ in terms of $|G:Z (G)|$ when $G/Z(G)$ does not have exponent $p$.  This yields Theorem \ref{fifth}.

\begin{theorem}\label{exponent quotient}
Let $(G,Z(G))$ be a Camina pair where $G$ is a $p$-group.  If $G/Z(G)$ has exponent $p^n$ with $n \ge 1$, then $|Z (G)|^n p^n \le |G:Z (G)|$.  In particular, $|Z (G)|^n < |G:Z (G)|$.
\end{theorem}

\begin{proof}
Since $G/Z(G)$ has exponent $p^n$, there exists an element $x \in G \setminus Z(G)$ so that $x^{p^{n-1}} \not\in Z(G)$ but $x^{p^n} \in Z(G)$.  We claim that $C (x^{p^{i+1}}) \ge D (x^{p^i})$ for $i = 0, \dots, n-1$.  Suppose that $a \in D (x^{p^i})$.  Then $[a,x^{p^i}] \in Z(G)$.  By Lemma \ref{exponent center}, we know that $[a,x^{p^i}]^p = 1$.  Since $[a,x^{p^i}]$ is central, we have $[a,x^{p^{i+1}}] = [a,x^{p^i}]^p = 1$.  It follows that $a \in C_G(x^{p^{i+1}})$.  By Lemma \ref{cents}, we know that $D (x^{p^i})/C (x^{p^i}) \cong Z (G)$ for $i = 0, \dots, n-1$.  Finally, we know that $x \in C (x)$.  It follows that $p^n$ divides $|C (x):Z (G)|$ and $|D (x^{p^i}):C (x^{p^i})| = |Z (G)|$ for $i = 0, \dots, n-1$.  Finally, we have
$$
|G:Z(G)| \ge \left(\prod_{i=0}^{n-1} |D (x^{p^i}):C (x^{p^i})| \right) |C (x):Z (G)| \ge \prod_{i=0}^{n-1} |Z (G)| p^n.
$$
It is not difficult to see that this implies that $|G: Z (G)| \ge |Z(G)|^n p^n.$
\end{proof}

%
%

We now apply Theorem \ref{exponent quotient} to $2$-groups.

\begin{corollary}
Let $(G,Z (G))$ be a Camina pair where $G$ is a $2$-group.  Then $|Z (G)|^2 \le |G:Z (G)|$.  Furthermore, if equality holds, then $G$ is a Camina group.
\end{corollary}

\begin{proof}
If $G' = Z (G)$, then $G$ is a Camina group, and we know that the result holds.  Thus, we may assume $G$ is not a Camina group and $Z (G) < G'$.  In particular, $G/Z (G)$ is not abelian. This implies that $G/Z (G)$ does not have exponent $2$.  By Theorem \ref{exponent quotient}, we know that $|Z (G)|^2 < |G:Z (G)|$.  This proves the corollary.
\end{proof}

In particular, this says that if $(G,Z (G))$ is a Camina pair where $|G:Z (G)| = 16$ and $G$ is not a Camina group, then $|Z (G)| = 2$.  Hence, $|G| = 32$.  Looking through the small groups library, we have found that there are 5 such groups.  In the small groups library in Magma they are ${\rm SmallGroup} (32,i)$ where $i = 6, 7, 8, 43, 44$.

Henceforth, we may assume that $p$ is odd and that $G/Z (G)$ has exponent $p$.

\section{Centralizers of $G/Z(G)$}

In this section, we consider the centralizers in $G/Z (G)$.  We begin with a sufficient condition for such a centralizer to be abelian.  In particular, we are working to show that $|Z (G)| \le |G:Z (G)|^{1/2}$ when $|G:Z (G)|$ is small.

\begin{lemma} \label{D-quo}
Let $(G,Z (G))$ be a Camina pair.  If $a \in G \setminus Z (G)$ satisfies $C (a) \cap G' = Z(G)$, then $D (a)/Z (G)$ is abelian.
\end{lemma}

\begin{proof}
Observe that $D (a)' \le C (a)$ since by Lemma \ref{cents} $D (a)/C (a)$ is isomorphic to an abelian group.  Obviously, $D (a)' \le G'$.  It follows that $D (a)' \le C (a) \cap G' = Z (G)$.  This proves the result.
\end{proof}

We note that is not difficult to see that $[Z_2,G'] = 1$.  In fact, $[Z_i,G_i] = 1$ for all $i$.  (This is Hauptsatz III.2.11 of Huppert's Endlichen Gruppen I.)  This can also be easily proved via 3 Subgroups Lemma.  Since $G' \le G' Z_2$, Theorem \ref{first} is a consequence of this lemma.

\begin{lemma} \label{Z_2G'}
Let $(G,Z (G))$ be a Camina pair.  If $G/Z (G)$ is not abelian, then $|G: G' Z_2| \ge |Z (G)|$.
\end{lemma}

\begin{proof}
Since $G/Z (G)$ is nilpotent and not abelian, we know that $G'/Z (G) \cap Z(G/Z (G)) > Z (G)/Z (G)$.  This implies that $G' \cap Z_2 > Z (G)$.  Hence, there exists $a \in G' \cap Z_2 \setminus Z(G)$.  We see that $D (a) = G$, so by Lemma \ref{cents}, we have $G/C (a) \cong Z (G)$.  Since $G'$ centralizes $Z_2$, we have $G' \le C (a)$, and since $Z_2$ centralizes $G'$, we have $Z_2 \le C (a)$.  This implies that $G' Z_2 \le C (a)$ and thus, $|G:G' Z_2| \ge |G:C (a)| = |Z (G)|$.
\end{proof}

We now use a condition on centralizers to bound $|Z(G)|$ in terms of $|G:Z (G)|$.

\begin{lemma} \label{index p}
Let $(G,Z (G))$ be a Camina pair.  If there exists $a \in G \setminus Z (G)$ so that $D (a)/Z (G)$ is abelian and $|G:D (a)| = p$, then $|Z (G)|^2 \le |G:Z (G)|$.
\end{lemma}

\begin{proof}
There exists $b \in G \setminus D (a)$.  Observe that $G = \langle b \rangle D (a) = D (b) D (a)$.  It follows that $D (a) \cap D (b)$ will centralize $G$ modulo $Z(G)$, so $D (a) \cap D (b) \le Z_2$.  Also, $|G:D (a) \cap D (b)| = |G:D (a)||D (a):D (a) \cap D (b)| = |G:D (a)||G:D (b)| = p|G:D (b)|$.  Observe that $|C (b):Z (G)| \ge p$.  Also, we may apply Lemma \ref{Z_2G'} to see that $|Z (G)| \le |G:Z_2|$.  Together, these equations and inequalities imply that $|Z (G)| \le |G:Z_2| \le |G:D (a) \cap D (b)| \le |G:D (b)||C (b):Z (G)|$.  By Lemma \ref{cents}, we know that $|D (b):C (b)| = |Z (G)|$, and we conclude that $|Z (G)|^2 \le |G:Z (G)|$.
\end{proof}

As a corollary, we obtain our bound when $|G':Z (G)| = p$.

\begin{corollary}
Let $(G,Z (G))$ be a Camina pair.  If $|G':Z (G)| = p$ where $p$ is a prime, then $|Z (G)|^2 \le |G:Z (G)|$.
\end{corollary}

\begin{proof}
Since $|G':Z (G)| = p$, we know that $G'/Z (G)$ is central in $G/Z (G)$.  For any element $g \in G$, the map $x \mapsto [g,x]Z(G)$ will be a homomorphism from $G$ to $G'/Z(G)$ whose kernel is $D (g)$.  This implies that $|G:D (g)| \le p$ for every $g \in G$.  Since $Z (G) < G'$, there exists $g \in G$ so that $C (g) \cap G' < G'$.  This implies that $C (g) \cap G' = Z (G)$ since $|G': Z(G)| = p$.  By Lemma \ref{D-quo}, $D (g)/Z(G)$ is abelian.  We are now done by Lemma \ref{index p}.
\end{proof}

This yields our conclusion when $|G:Z (G)| \le p^6$.

\begin{corollary} \label{m=2}
Let $(G,Z(G))$ be a Camina pair where $G$ is a $p$-group for some prime $p$.  Then either $|Z(G)|^2 \le |G:Z(G)|$ or $|Z(G)|p^3 \le |G:Z(G)|$.
\end{corollary}

\begin{proof}
We may assume $Z(G) < G'$.  Let $a \in G \setminus Z_2 G'$.  Thus, $|G:D(a)| \ge p$ and $|C(a):Z(G)| \ge p$, and so, $|G:D(a)||C(a):Z(G)| \ge p^2$.  By Lemma \ref{cents}, we know that $|D(a):C(a)| = |Z(G)|$.  If $|G:D(a)||C(a):Z(G)| \ge p^3$, then $|G:Z(G)| \ge p^3 |Z(G)|$.  If $|G:Z(G)| \le p^6$, then this implies that $|Z(G)| \le p^3$, and so, $|Z(G)|^2 \le |G:Z(G)|$. Thus, we may assume that $|G:D(a)||C(a):Z(G)| = p^2$.  This implies that $|G:D(a)| = |C(a):Z(G)| = p$, so $C(a) = \langle a \rangle Z(G)$.  Since $a \not\in G'$, we have $C(a) \cap G' = Z(G)$.  By Lemma \ref{D-quo}, $D(a)/Z(G)$ is abelian.  Applying Lemma \ref{index p}, we obtain the conclusion that $|Z(G)|^2 \le |G:Z(G)|$.
\end{proof}

Let $(G, Z(G))$ be a Camina pair.  We define the set ${\mathcal C} (G)$ by $\{ x \in G \mid C (x) \cap G' > Z (G) \}$.  (We normally omit the $G$, and just write ${\mathcal C}$.)  If $G' = Z(G)$, then ${\mathcal C}$ is empty.  On the other hand, it is not difficult to see that $G'$ and $C (G')$ are contained in ${\mathcal C}$, so ${\mathcal C}$ is nonempty if $Z(G) < G'$.  This next lemma gives a second characterization of ${\mathcal C}$.

\begin{lemma}\label{scriptC}
Let $(G,Z (G))$ be a Camina pair with $Z (G) < G'$.  Then ${\mathcal C} = \cup_{a \in G' \setminus Z (G)} C (a)$.
\end{lemma}

\begin{proof}
Suppose $x \in {\mathcal C}$.  Then $C (x) \cap G' > Z (G)$.  Hence, there exists $b \in (C (x) \cap G') \setminus Z (G)$.  Observe that $x \in C (b) \subseteq \cup_{a \in G' \setminus Z (G)} C (a)$.  Hence, ${\mathcal C} \subseteq \cup_{a \in G' \setminus Z (G)} C (a)$.

On the other hand, if $x \in \cup_{a \in G' \setminus Z (G)} C (a)$, then $x \in C (b)$ for some $b \in G' \setminus Z (G)$.  It follows that $b \in (C(x) \cap G') \setminus Z (G)$.  This implies that $Z (G) < C (x) \cap G'$, and so, $x \in {\mathcal C}$.
\end{proof}


We now use the structure of $\mathcal C$.

\begin{lemma}\label{inter}
Let $(G,Z (G))$ be a Camina pair with $Z (G) < G'$.  Let $n$, $m$, and $l$ be positive integers so that $|G:Z(G)| = p^n$, $|Z(G)| = p^{m}$, and $|G':Z(G)| = p^l$.  Suppose $k$ is in integer so that $p^k \le |G:H|$ for any subgroup $H$ of $G$ all of whose elements lie in ${\mathcal C}$.  If ${\mathcal C} < G$, then $k \le 2l$ and $k \le 2(n-m-1)$.
\end{lemma}

\begin{proof}
If $x \in G \setminus {\mathcal C}$, then $C (x) \cap G' = Z(G)$.  By Lemma \ref{D-quo}, $D (x)/Z (G)$ is abelian.  Thus, we may pick $x \in G \setminus Z(G)$ so that $D(x)$ is maximal subject to $D(x)/Z(G)$ is abelian.  Let $y \in G \setminus D(x)$.  For $z \in D(x) \cup D(y)$, we have that $D (x)$ and $y$ are both in $D (z)$, so $D (x) < D (z)$.  The maximality of $D (x)$ implies that $D (z)/Z(G)$ is nonabelian.  By \ref{D-quo}, this implies that $C (z) \cap G' > Z (G)$, and so, $z \in {\mathcal C}$.  In particular, every element in $D(x) \cap D (y)$ lies in ${\mathcal C}$.  By the hypothesis, we have $|G:D (x) \cap D (y)| \ge p^k$.

Notice that $|G:D (x)||C (x):Z (G)| = |G:D (y)||C (y):Z (G)| = p^{n-m}$.  Since $x \in C (x) \setminus Z (G)$ and $y \in C (y) \setminus Z (G)$, we have $|G:D (x)| \le p^{n-m-1}$ and $|G:D (y)| \le p^{n-m-1}$.  This implies that $|G: D (x) \cap D (y)| \le |G:D (x)||G:D (y)| \le p^{2n-2m-2}$.  Since $D (x) \cap D(y) \le C (a)$, we have $|G: C (a)| \le |G:D (x) \cap D (y)|$.  This implies that $p^{k} \le p^{2n-2m-2}$, and hence, $k \le 2(n - m - 1)$.

Also, notice that $|G:D (x)|$ is the size of the conjugacy class of $x Z (G)$ in $G/Z (G)$.  It is not difficult to see that the size of the conjugacy class of $x Z (G)$ in $G/Z (G)$ will be bounded by $|G':Z(G)|$.  Thus, $|G:D (x)| \le p^l$.  It follows that $|G: D (x) \cap D (y)| \le p^{2l}$, and working as in the last paragraph, we obtain $k \le 2l$.
\end{proof}

We get a better bound in a restricted case.

\begin{corollary} \label{small}
Let $(G,Z(G))$ be a Camina pair with $Z (G) < G'$.  Let $n$ and $m$ be positive integers so that $|G:Z (G)| = p^n$ and $|Z (G)| = p^{m}$.  If there exists $a \in (G' \cap Z_2) \setminus Z (G)$ such that $C (b) \le C (a)$ for all $b \in G' \setminus Z (G)$, then $3m + 2 \le 2n$.  In particular, this occurs if either $|C_G (G'): Z(G)| = p^{n-m}$ or $|G':Z(G)| = p$.
\end{corollary}

\begin{proof}
Suppose there exists $a \in (G' \cap Z_2) \setminus Z (G)$ such that $C (b) \le C (a)$ for all $b \in G' \setminus Z (G)$.  It follows that ${\mathcal C} = C (a)$.  Since $C (a) < G$, this implies that ${\mathcal C} < G$.
Hence, in the notation of Lemma \ref{inter}, we have $k = m$ since $|G:C (a)| = |Z (G)|$.  By that lemma, we obtain $m \le 2(n - m - 1)$, and so, $3m + 2 \le 2n$.

If $|G':Z(G)| = p$, then $G' = \langle a, Z (G) \rangle$.  Observe that $a \in Z_2$ and if $b \in G' \setminus Z(G)$, then $C (b) = C (a)$.  Suppose now that $|C_G (G'):Z (G)| = p^{n-m}$.  If $b \in G' \setminus Z (G)$, then $C_G (G') \le C (b)$ and $|C (b):Z (G)| \le p^{n-m}$.  It follows that $C_G (G') = C (b)$.  Taking $a \in (G' \cap Z_2) \setminus Z (G)$, we have $C (b) = C_G (G') = C (a)$.
\end{proof}

\section{Bounding $|Z(G)|$ by $|G':Z (G)|$}

In this next theorem, we show that $|Z (G)|$ can be bound by a function in terms of $|G':Z (G)|$.  This yields Theorem \ref{second}.

\begin{theorem}\label{derived bound}
Let $(G,Z (G))$ be a Camina pair.  Suppose $m$ and $l$ are positive integers so that $|Z (G)| = p^m$ and $|G':Z (G)| = p^l$.  Then $m \le 3l - 1$.
\end{theorem}

\begin{proof}
Observe that $l \ge 1$.  Thus, if $m \le l$, then $m < 3l - 1$.  Hence, we may assume that $m > l$.  If ${\mathcal C} = C (a)$ for some $a \in G' \setminus Z (G)$, then in the notation of Lemma \ref{inter}, we have $k = m$, and by the conclusion of that lemma, $m \le 2l \le 3l - 1$.  Thus, we may assume that $C (a)$ is properly contained in ${\mathcal C}$ for all $a \in G' \setminus Z (G)$.  This implies that $C (G') < C (a)$ when $a \in (G' \cap Z_2) \setminus Z (G)$.

We write ${\mathcal S} = \{ C (a) \mid a \in G' \setminus Z (G) \}$ for the set of centralizers so that ${\mathcal C} = \cup_{C \in {\mathcal S}} C$.  Notice that if $a, a' \in G' \setminus Z (G)$ with $\langle a, Z(G) \rangle = \langle a', Z(G) \rangle$, then $C (a) = C (a')$.  Hence, $|{\mathcal S}|$ is bounded by the number of cyclic subgroups in $G'/Z (G)$.  Since $|G':Z (G)| = p^l$, the number of cyclic subgroups in $G'/Z (G)$ is $(p^l - 1)/(p-1)$.  In particular, $|{\mathcal S}| < (p^l - 1)$.  Let $n$ be the positive integer so that $|G:Z (G)| = p^n$.  It follows that $|G| = |G:Z (G)||Z (G)| = p^n p^m = p^{n+m}$.  If $a \in G \setminus Z(G)$, then $|G:C (a)| \ge p^m$, so $|C (a)| \le p^n$.  In particular, if $C \in {\mathcal S}$, then $|C| \le p^n$.  If we let $|C (G')| = p^h$, then $h < n$.

We now have $|{\mathcal C}| = |\cup_{C \in {\mathcal S}} C| = |\cup_{C \in {\mathcal S}} C \setminus C (G')| + |C (G')|$.  This implies that
$$
|{\mathcal C}| \le \sum_{C \in {\mathcal S}} \left(|C| - |C (G')|\right) + |C (G')| \le |{\mathcal S}|(p^n - p^h) + p^h.
$$
We then obtain $|{\mathcal C}| < (p^l - 1)(p^n - p^h) + p^h.$

Multiplying out the expression in the right-hand side of the inequality, we obtain $(p^l - 1)(p^n - p^h) + p^h = p^l p^n - p^n - p^l p^h + p^h + p^h < p^{l+n}$.  This shows that $|{\mathcal C}| < p^{l+n}$.  Recall that $l < m$, so $|{\mathcal C}| < p^{m+n} = |G|$.  If $H$ is a subgroup of $G$ so that every element of $H$ lies in ${\mathcal C}$, then $|H| \le |{\mathcal C}| < p^{n+l}$.  It follows that $|H| \le p^{n+l-1}$.  This implies that $|G:H| \ge p^{n+m}/p^{n+l-1} = p^{m-l+1}$.  In the context of Lemma \ref{inter}, we may take $k = m-l+1$.
Applying that lemma, we obtain $m - l + 1 \le 2l$.  We conclude that $m \le 3l - 1$.
\end{proof}

We now use the bound of the previous theorem to obtain a bound on $|Z (G)|$ in terms of $|G:Z (G)|$.  In particular, this is Theorem \ref{third}.

\begin{corollary}
Let $(G, Z(G))$ be a Camina pair.  If $|G:Z (G)| = p^n$ and $|Z (G)| = p^m$, then $4m + 1 \le 3n$.
\end{corollary}

\begin{proof}
Let $l$ be an integer so that $|G':Z (G)| = p^l$.  By Theorem \ref{derived bound}, we know that $m \le 3l - 1$.  By Lemma \ref{Z_2G'}, we have $|G:G'| \le p^m$.  It follows that $|G':Z (G)| \le p^{n-m}$.  It follows that $(m + 1)/3 \le l$ and $l \le n - m$.  Combining these, $(m+1)/3 \le n - m$, and so, $m + 1 \le 3n - 3m$.  Finally, we deduce that $4m + 1 \le 3n$.
\end{proof}

When $|G:Z (G)|$ is small, we can bound $|Z (G)|$ by $|G:Z (G)|^{1/2}$.  This is Theorem \ref{fourth}.

\begin{theorem}
Let $(G,Z(G))$ be a Camina pair with $Z (G) < G'$.  Then either $|Z (G)|^2 \le |G:Z(G)|$ or $|Z (G)|p^4 \le |G:Z(G)|$.  In particular, if $|G:Z(G)| \le p^8$, then $|Z (G)|^2 \le |G:Z (G)|$.
\end{theorem}

\begin{proof}
Let $n$ and $m$ be positive integers so that $|G:Z(G)| = p^n$ and $|Z(G)| = p^{m}$.  We need to show that either $n - m \ge 4$ or $n \le n/2$.  We assume $n - m \le 3$.  We know by Lemma \ref{m=2} that the result holds when $n - m = 2$.  Thus, we may assume $n - m = 3$.  If the inequality $3m + 2 \le 2n$ holds, then $n + 2 \le 3n - 3m = 3(n - m)$ holds by adding $n - 3m$ to both sides.  Since $n - m = 3$, we have $n \le 7$.  On the other hand, we know $|G:Z (G)|$ has to be a square, so $n$ is even.  Hence, $n \le 6$, and either $n = 6$ or $n = 4$.  If $n = 6$, then $m = 3 = n/2$, and if $n = 4$, then $m  = 1 < n/2$, and so the result holds.  Thus, it suffices to prove that $3m + 2 \le 2n$ holds.

By Lemma \ref{Z_2G'}, $|G':Z(G)| \le p^{n-m} = p^3$.  If $|G':Z(G)| = p$, then we are done by applying Corollary \ref{small}.  Thus, we may assume that $|G':Z (G)| > p$.  Notice that $Z_2 \le C_G (G')$.  If $|Z_2:Z (G)| = p^3$, then we may again Corollary \ref{small} to obtain the conclusion.  Thus, we may assume that $|Z_2:Z (G)| < p^3$.

We now suppose that $|G':Z (G)| = p^3$.  By the previous paragraph, we have $G' \cap Z_2 < G'$.  If $|G':G' \cap Z_2| = p$, then $G' \cap Z_2$ is central in $G'$ and $G'/G' \cap C_2$ is cyclic.  It follows that $G'$ is abelian, and so, $G' \le C_G (G')$.  We deduce that $|C_G (G'):G'| = p^3$, and we may apply Corollary \ref{small} once again to obtain the conclusion.  Thus, we must have $|G':G' \cap Z_2| = p^2$, and so, $|G' \cap Z_2: Z (G)| = p$.  This implies that $G' \cap Z_2 = \langle a, Z(G) \rangle$.  If $x \in (Z_2 \cap G') \setminus Z(G)$, then $C (x) = C(a)$.  Recall that $G' \le C (a)$, and since $|G':Z (G)| = |C (a): Z(G)| = p^3$, we conclude that $C (a) = G'$.  If $b \in G' \setminus Z_2$, then $|G:D (b)| \ge p$, so $|C (b):Z (G)| \le p^2$.  Since $a, b \in C (b)$, we conclude that $C (b) = \langle a, b, Z (G) \rangle \le G'$.  This implies that the hypotheses of Corollary \ref{small} are met, and so the conclusion holds.

For the rest of this proof, we have $|G':Z (G)| = p^2$.  If $G' \cap Z_2 < G'$, then $|G' \cap Z_2: Z(G)| = p$.  In particular, $Z_2 \cap G' = \langle a, Z(G) \rangle$.  If $x \in (Z_2 \cap G') \setminus Z(G)$, then $C (x) = C(a)$.  Let $b \in G' \setminus Z_2$.  Then $|G:D (b)| \ge p$, so $|C (b): Z (G)| \le p^2$.  Since $a, b \in C (b)$, we have $C (b) = \langle a, b, Z(G) \rangle \le G' \le C (a)$.  We now apply Corollary \ref{small}, and we are done.

We have $G' \cap Z_2 = G'$, and since $|Z_2:Z (G)| < p^3$, we have $Z_2 = G'$.  Hence, $G' \le C_G (G')$.  If $G' < C_G (G')$, then we may appeal to Corollary \ref{small} to finish, so we may assume $G' = C_G (G')$.  We can find $a_1, a_2$ so that $G' = \langle a_1, a_2, Z(G) \rangle$.  Notice that $G' = C_G (G') = C_G (a_1) \cap C_G (a_2)$.

Consider $a \in G' \setminus Z (G)$.  We know that $G = D (a)$, so $|C (a): Z (G)| = p^3$.  Since $G' \le C (a)$, we have $|C (a):G'| = p$.  It follows that $C (a) = \langle b, G' \rangle$ for some $b \in G \setminus Z_2$.  Since $b$ is not in $Z_2$, it follows that $|G:D (b)| \ge p$, and so, $|C (b): Z(G)| \le p^2$.  Now, $a, b \in C (b)$, and so, $C (b) = \langle a, b, Z (G) \rangle$.  This implies that $|C (b):Z (G)| = p^2$, and so, $|G:D (b)| = p$.  By Lemma \ref{index p}, we know that $D (b)/Z (G)$ is not abelian.

Now, pick $b_i$, for $i = 1, 2$ so that $C (a_i) = \langle a_i, b_i, Z (G) \rangle$.  By Lemma \ref{cents}, we know that $D (b_i)' \le C (b_i)$.  Let $L = D (b_1) \cap D (b_2)$, and observe that $L' \le D (b_1)' \cap D (b_2)' \le C (b_1) \cap C (b_2)$.  Notice that $C (b_i) \le C (a_i)$, so $C (b_1) \cap C (b_2) \le C (a_1) \cap C (a_2) = G'$.  Hence, $C (b_1) \cap C (b_2) = C (b_1) \cap C (b_2) \cap G' = Z (G)$.  It follows that $L/Z (G)$ is abelian.  Notice that if $D (b_1) = D (b_2)$, then $D (b_1) = L$ which is a contradiction, since we have that $D (b_1)/Z (G)$ is not abelian.  Hence, we have that $D (b_1) \ne D (b_2)$.  Since $|G:D (b_1)| = |G:D (b_2)| = p$, it follows that $|G:L| = p^2$.

If $b_1$ is not in $L$, then $b_1$ is not in $D (b_2)$.  This implies that $b_2$ is not in $D (b_2)$, and so, $b_2$ not in $L$.  Notice that $D (b_1) = \langle b_1, L \rangle$ and $D (b_2) = \langle b_2, L \rangle$.  It follows that $G = \langle b_1, b_2, L \rangle$.  Now, if $x \in L$, then $L, b_1, b_2 \in D (x)$, and thus, $D (x) = G$.  This implies that $x \in Z_2$, and so, $L \le Z_2$.  We have $p^2 = |G:L| \ge |G:Z_2| \ge p^{m}$.  We obtain $m \le 2$, and since $n - m = 3$. This implies $n \le 5$.  We have seen that $n$ is even and at least $4$.  We conclude that $n = 4$ and $m = 2$, and the result holds.  Thus, we may assume that $b_1 \in L$ and this implies that $b_1 \in D (b_2)$, and so, $b_2 \in D (b_1)$.  We conclude that $b_2 \in L$.  Notice that this implies that $C (a_1), C (a_2) \le L$.

Let $a \in G' \setminus Z (G)$ and consider $b \in C_G (a)$.  Suppose $a \in \langle a_1, Z (G) \rangle \cup \langle a_2, Z (G) \rangle$.  By the previous paragraph, we know that $b \in L$.  We now suppose that $a \in G' \setminus (\langle a_1, Z (G) \rangle \cup \langle a_2, Z (G) \rangle)$.  It follows that $G' = \langle a, a_1, Z (G) \rangle = \langle a, a_2, Z(G) \rangle$.  Applying the argument in the previous paragraph, we see that $b \in D (b_1)$ and $b \in D (b_2)$.  It follows that $b \in D (b_1) \cap D (b_2) = L$.  In particular, we have $L \le D (b)$.  If $b \not\in G' = Z_2$, then we have shown that $|G:D (b)| = |D (b):L| = p$.  Notice that the previous paragraph implies that $D(b) \cap D(b_1) = D(b) \cap D(b_2) = L$.

We now show that there is a bijection between ${\mathcal Z} = \{ \langle a, Z (G) \rangle \mid a \in G' \setminus Z (G) \}$ and ${\mathcal L} = \{ \langle c, L \rangle \mid c \in G \setminus L \}$ defined by $\langle a, Z(G) \rangle \mapsto M_a$ where $M_a = \langle c, L \rangle = D (b)$ with $C (a) = \langle b, G' \rangle$.

We first show that this map is well-defined.  Observe that if $\langle a, Z(G) \rangle = \langle a', Z(G) \rangle$, then $C (a') = C (a)$.  Next, observe that if $C (a) = \langle b, G' \rangle = \langle b', G' \rangle$, then $b' = b y$ where $y \in G' = Z_2$.  This implies that $D (b') = D (by) = D (b)$.  This shows that the map is well-defined.

We next show that the map is one-to-one.  Suppose that $\langle a, Z (G) \rangle \ne \langle a', Z(G) \rangle$.  Then $G' = \langle a, a', Z(G) \rangle$, and we can repeat the previous paragraph with $a$ and $a'$ in place of $a_1$ and $a_2$.  From that paragraph, we know that $M_a \cap M_{a'} = L$, so $M_a \ne M_{a'}$.  This shows that the map is one-to-one.  Notice that $|G':Z(G)| = |G:L| = p^2$.  Hence, $|{\mathcal Z}| = |{\mathcal L}| = (p^2 - 1)/(p - 1) = p+1$ (the number of subgroups of order $p$ in an elementary abelian group of order $p^2$).  Thus, the map is a bijection.

Suppose $c \in L$.  Since $L/Z(G)$, we know that $L \le D (c)$. If $c \not\in G' = Z_2$, then we know that $D (c) < G$.  If $L < D (c)$, then $D (c) = M_a$ for some $a \in G' \setminus Z (G)$.  By Lemma \ref{index p}, we know that $C (c) \cap G' > Z (G)$.  Hence, there exists $a' \in (C (c) \cap G') \setminus Z (G)$, and so, $c \in C (a')$.  It follows that $M_{a'} = D (c) = M_a$.  By the bijection, this implies that $\langle a, Z(G) \rangle = \langle a', Z(G) \rangle$, and so, $C (a) = C (a')$. In particular, $c \in C (a)$.

Observe that $|L:Z_2| = p^{n-4}$.  If $n = 4$, then the result holds.  Thus, we may assume that $n > 4$ and so $Z_2 < L$.  Thus, there exists an element $g \in G$ with $D (g) \cap L < L$.  If $g \in L$, then $L \le D (g)$ since $L/Z(G)$ is abelian.  Thus, $g$ is not in $L$.  We have $\langle g, L \rangle \in {\mathcal L}$.  Thus, $\langle g, L \rangle = M_a$ for some $a \in G' \setminus Z (G)$.  If $c \in D (g) \cap L$, then $M_a = \langle g, L \rangle \le D (c)$.  By the previous paragraph, this implies that either $c \in Z_2$ or $c \in C (a)$.  Since $Z_2 \le C (a)$, it follows that $D (g) \cap L \le C (a)$.  We know that $|G: D(g) \cap L| \le |G:D (g)||G:L| \le p^2 p^2 = p^4$.  We now have $p^{m} = |G:C (a)| \le |G:D (g) \cap L| = p^4$ and $m \le 4$.  It follows that $n = m + 3 \le 7$.  Since $n$ is even, $n = 6$ and thus, $m = 3$, and we have the desired result.
\end{proof}

\section{Examples?}

The following are some properties of a possible family of groups which if the family exists would be a counterexample to the conjecture.  We remind the reader that we do not have any examples of Camina pairs $(G,Z(G))$ where $|Z(G)|^2 > |G:Z(G)|$.  What follows is an outline of a possible family which if it is exists would have this property.  In talking with James Wilson, we believe it is possible to prove that these groups exist, but we have not included the proof here.

Take $p$ to be an odd prime and $k$ to be an (odd) integer.  Then $G$ will be a class 3 group of order $p^{5k+1}$.  Both $G/Z(G)$ and $Z(G)$ have exponent $p$.  (I would not be surprised if $G$ has exponent $p$, but that does not seem to be required.)  The center of $G$, $Z(G)$ has order $p^{2k}$.  Thus, $Z (G)$ is an elementary abelian $p$-group of order $p^{2k}$.  We have $Z(G) = [G',G]$.  Also, $Z_2 (G)$ is an abelian group of order $p^{3k+1}$.  In other words, $Z_2 (G)/ Z (G)$ is an elementary abelian group of order $p^{k+1}$.  If $g \in Z_2 (G) \setminus Z(G)$, then $C_G (g) = Z_2 (G)$.  We require $G/Z_2 (G)$ to be abelian.  Notice that $|G:Z_2 (G)| = p^{2k}$.  I can show that $|Z_2 (G):G'| \le p$.  It seems that either $Z_2 (G) = G'$ or $|Z_2 (G):G'| = p$ can occur.  At this time, I do not see any reason why one of these possibilities must occur or not occur. If $g \in G \setminus Z_2 (G)$, then $C_G (g) = \langle g, Z(G) \rangle$.  In addition, $|G:D(g)| = p^k$, so $|D(g):Z_2 (G)| = p^k$.  We need $D (g)/ Z(G)$ is abelian, and in fact, $D (g)' = Z(G)$.  Also, if $h \in D (g) \setminus Z_2 (G)$, then $D (g) = D (h)$.  In particular, $G \setminus Z_2 (G)$ is partitioned by the sets $D (g) \setminus Z_2 (G)$ as $g$ runs over the elements in $G \setminus Z_2 (G)$.

Obviously, $G/ Z(G)$ is a class $2$ group with exponent $p$ whose center has order $p^{k+1}$ and index $p^{2k}$.  The centralizer of every noncentral element is abelian and has order $p^{2k+1}$.  Such groups definitely occur.  Let $S$ be a Sylow $p$-subgroup of ${\rm SL}_3 (p^k)$.  Then $T = S \times Z_p$ has exactly these properties.  Observe that $Z(T) = T' \times Z_p$ where $T' = S'$, so $|Z(T):T'| = p$.  Do there exist examples of such groups with equality?

Notice that if $g \in G \setminus Z_2 (G)$, then $D (g)$ has class $2$.  In fact, we claim that $D (g)$ is special.  I.e., we claim that $Z (D(g)) = Z (G)$.  Notice that $Z_2 (G)$ is an abelian, characteristic subgroup of $D (g)$ with index $p^k$ and order $p^{3k + 1}$.  If $h \in D(g) \setminus Z_2 (G)$, then $C_{D (g)} (h) = \langle h, Z (G) \rangle$, and so $|D (g): C_{D (g)} (h)| = p^{2k} = |D (g)'|$.  This implies that ${\rm cl}_{D(g)} (h) = h Z (G)$.

Notice that if these groups exist, then we have $|Z (G)| = p^{2k}$ and $|G:Z(G)| = p^{3k+1}$.  In particular, these groups have the property that $|Z (G)| < |G:Z (G)|^{2/3}$, and so, even if these groups exist, there would still be a gap between the bound we can prove and the bound that can be obtained (at least asymptotically).


\bigskip

\begin{thebibliography}{99}
\bibitem{Camina}  A.~R.~Camina, Some conditions which almost characterize Frobenius groups, {\it Israel J. Math.} {\bf 31} (1978), 153-160.
\bibitem{ChMc}    D.~Chillag and I.~D.~MacDonald, Generalized                  Frobenius groups, {\it Israel J. Math.} {\bf 47} (1984), 111-122.
\bibitem{ChMaSc}  D.~Chillag, A.~Mann, and C.~M.~Scoppola,
     Generalized Frobenius groups II, {\it Israel J. Math.} {\bf 62} (1988), 269-282.
\bibitem{DaSc}    R.~Dark and C.~M.~Scoppola, On Camina groups of prime power order, {\it J. Algebra} {\bf 181} (1996),
    787-802.
\bibitem{text}    I.~M.~Isaacs, Character Theory of Finite Groups, AMS Chelsea Publishing, Providence, RI, 2006.
\bibitem{coprime} I.~M.~Isaacs, Coprime group actions fixing all nonlinear irreducible characters, {\it Can. J. Math.} {\bf 41} (1989), 68-82.
\bibitem{sylow}   E.~B.~Kuisch, Sylow $p$-subgroups of solvable                  Camina pairs, {\it J. Algebra} {\bf 156} (1993), 395-406.
\bibitem{MacD1}   I.~D.~Macdonald, Some $p$-groups of Frobenius and extra-special type, {\it Israel J. Math.} {\bf 40} (1981), 350-364.
\bibitem{more}    I.~D.~Macdonald, More on $p$-groups of Frobenius type, {\it Israel J. Math.} {\bf 56} (1986), 335-344.
\end{thebibliography}
\end{document}